\pdfoutput=1
\documentclass[11pt]{amsart}
\usepackage{amsmath,amsthm,amscd,amssymb, color}
\usepackage{latexsym}
\usepackage[colorlinks, citecolor = blue]{hyperref}
\usepackage[alphabetic]{amsrefs}
\usepackage[capitalize, nameinlink]{cleveref}
\usepackage{enumerate}
\usepackage[margin=1.25in]{geometry}
\newcommand{\Q}{\mathbb Q}

\newcommand{\Z}{\mathbb Z}

\renewcommand{\phi}{\varphi}
\newcommand{\eps}{\varepsilon}
\newcommand{\frakp}{\mathfrak p}

\newcommand{\frako}{\mathfrak o}

\newcommand{\bmx}{\left( \begin{matrix}}
\newcommand{\emx}{\end{matrix} \right)}

\newcommand{\End}{\mathrm{End}}

\usepackage{bbm}

\DeclareMathOperator{\GL}{GL}

\DeclareMathOperator{\Ind}{Ind}

\newtheorem{lem}{Lemma}
\numberwithin{lem}{section}
\newtheorem{prop}[lem]{Proposition}
\newtheorem{thm}[lem]{Theorem}
\newtheorem{cor}[lem]{Corollary}

\theoremstyle{remark}
\newtheorem{rem}[lem]{Remark}

\theoremstyle{definition}

\numberwithin{equation}{section}

\begin{document}

%\title{Local conductors and rationality fields of modular forms}
\title{Local conductor bounds for modular abelian varieties}
\author{Kimball Martin}
\email{kimball.martin@ou.edu}
\address{Department of Mathematics, University of Oklahoma, Norman, OK 73019 USA}
\thanks{The author was partially supported by the Japan Society for the 
Promotion of Science (Invitational Fellowship L22540), and the
Osaka Central Advanced Mathematical Institute (MEXT Joint Usage/Research Center on Mathematics and Theoretical Physics JPMXP0619217849).}

\date{\today}

\begin{abstract}
Brumer and Kramer gave bounds on local conductor exponents for an
abelian variety $A/\Q$ in terms of the dimension of $A$ and the localization prime
$p$.  Here we give improved bounds in the case that $A$ has 
maximal real multiplication, i.e., $A$ is isogenous to a factor of the
Jacobian of a modular curve $X_0(N)$.  In many cases, these bounds are
sharp.  
The proof relies on showing that the rationality field of a newform for 
$\Gamma_0(N)$, and thus the endomorphism algebra of $A$, 
contains $\Q(\zeta_{p^r})^+$ when $p$ divides $N$ to
a sufficiently high power.  We also deduce that certain divisibility conditions
on $N$ determine the endomorphism algebra when $A$ is simple.
\end{abstract}

\maketitle

%%%%%%%%%%%%%%%%%%%%%%%%%%%%%%
%%%%%%%%%%%%%%%%%%%%%%%%%%%%%%

\section{Introduction}

Let $A/\Q$ be a $d$-dimensional abelian variety of conductor $N_A$.  
Brumer and Kramer \cite{brumer-kramer} gave upper bounds $v_p(N_A) \le B(p,d)$
on local conductor exponents  in terms of $p$ and $d$.  
Moreover, their bounds are sharp in the sense for any $p, d$, there exists
an abelian variety $A$ with $v_p(N_A) = B(p,d)$.
The abelian variety $A$ is said to have real multiplication 
(RM) if the endomorphism algebra $\End^0(A) := \End(A) \otimes \Q$ 
contains a totally real number field  $K \supsetneq \Q$.
Brumer and Kramer suggested that stronger bounds on local
conductor exponents may exist when restricting to abelian varieties
with RM.

We say $A$ has maximal RM if $\End^0(A)$ contains a totally
real field $K$ of degree $d = \dim A$. (Necessarily, $[K : \Q] \mid d$
for any subfield $K$ of $\End^0(A)$.)
 Abelian varieties with maximal RM are of GL(2)-type, 
 and they are simple over $\Q$ if and only if
$\End^0(A) \simeq K$ is a totally real field of degree $d$ \cite[Theorem 2.1]{ribet:korea}.  
Up to isogeny, the simple abelian varieties with maximal RM are precisely
the simple factors of Jacobians $J_0(N) = \mathrm{Jac}(X_0(N))$ of the modular curves $X_0(N)$.  If $N$ is minimal such that $A$ is isogenous
to a simple factor of $J_0(N)$, then $N_A = N^d$.

We improve the Brumer--Kramer bounds for
abelian varieties with maximal RM.  Define
\begin{equation} \label{eq:B0def}
 B_0(p,d) = \begin{cases}
8 + 2 v_2(d) & \text{if } p = 2, \\
5 + 2 v_3(d) & \text{if } p = 3, \\
4 + 2 v_p(d) & \text{if } p \ge 5 \text{ and } (p - 1) \mid 2d, \\
2 &\text{else}.
\end{cases} 
\end{equation}

\begin{thm} \label{thm:main}
Let $A/\Q$ be a $d$-dimensional simple abelian variety with maximal RM and conductor $N^d$. 
\begin{enumerate}
\item We have $v_p(N) \le B_0(p,d)$, i.e., $v_p(N_A) \le d B_0(p,d)$.
%\item Suppose $d \ge 2$ and let $\Sigma_d(N)  $v_p(N) = B_0(p,d)$ is not simultaneously attained for all $p \in  \{ 2, 3 \} \cup \{ p \ge 5: (p - 1) \mid 2d   \}$.
\item The bounds in (1) are stronger than what the Brumer--Kramer bounds imply.  Namely, $B_0(p,d) \le \lfloor \frac{B(p,d)}d \rfloor$ for all $p, d$,
and this is a strict inequality if either (a) $5 \le p < 2d + 1$ and $(p-1) \nmid 2d$,
or (b) $p \le 3$, $d > 3$ and $p \nmid d$.
It is an equality when $p \ge 2d+1$.
(In other cases, this is sometimes an equality and sometimes not.)
\item The bound in (1) sharp, i.e., $v_p(N) = B_0(p,d)$  occurs
for some such $A$, for all $(p,d)$ such that $d \le 10$, with the possible 
exclusion of the following 5 cases: $B_0(5,10) = 6$, $B_0(11,10) = 4$,
and $B_0(2d+1,d)=4$ for $d = 6,7,8$.
\end{enumerate}
\end{thm}

\begin{rem} \label{rem:GL2}
 If $A$ is a $d$-dimensional simple abelian variety of 
$\GL(2)$-type, i.e., it is isogenous to a factor of some $J_1(N) =
\mathrm{Jac}(X_1(N))$,
 then one similarly has the bounds $v_p(N_A) \le d B_0(p,d)$
for $p$ odd, and $v_2(N_A) \le d (B_0(2,d) + 1)$.  This improves
the Brumer--Kramer bounds for odd $p$, and for certain values of
$d$ when $p=2$.  See \cref{sec:GL2}.
\end{rem}

\begin{rem}
If $A$ is not simple, then applying the above bounds to its simple
factors yield even stronger bounds in terms of $d$.
E.g., suppose $A \simeq A_1 \times A_2$ where $A_1$ and $A_2$
are isogenous simple abelian varieties with endomorphism algebra
$K_0$, a totally real field of degree $d/2$.  Then 
$\End^0(A) \simeq M_2(K_0)$, which contains totally real fields of
degree $d$, and thus $A$ has maximal RM.  Since $N_A = N_{A_1}^2$,
one sees that $v_p(N_A) \le 2 v_p{N_{A_1}} \le d B_0(p,d/2)$.  This is
better than the bound $v_p(N_A) \le d B_0(p,d)$ for simple $A$.
\end{rem}

The precise formula for the Brumer--Kramer bounds is slightly
complicated---see \cref{sec:BK} for details---but we list the bounds $B'(p,d) := \lfloor \frac{B(p,d)}d \rfloor$ (i.e., the Brumer--Kramer bounds
applied to $v_p(N)$ with $N$ and $A$ as in \cref{thm:main}) for $d \le 10$
in \cref{tab:bounds}.  For $p > 2d+1$, $B'(p,d) = B_0(p,d) = 2$, and we omit these
entries.  We write the bound $B_0(p,d)$ from \cref{thm:main}
in parentheses when it is smaller.  

In this table, we bolded all of the cases where we have checked the upper 
bound is sharp, by finding associated modular forms.  We also starred
the cases $B_0(19,9)$ and $B_0(11,10)$ to indicate the upper bounds are
at least ``almost sharp'', in the sense that $v_p(N) = B_0(p,d) - 1$ is attained.
Note that by quadratic twisting at $p$, an upper bound of 
$B_0(p,d) = 2$ is always sharp, provided there exists at least one simple 
$d$-dimensional abelian variety $A/\Q$ with maximal RM.
It seems plausible that the bound $B_0(p,d)$ is always sharp, but we do not
know of constructions of simple abelian varieties with maximal RM (or more
generally of GL(2) type) in higher
dimensions which would shed light on this.

\begin{table} 
\begin{tabular}{r|l|l|l|l|l|l|l|l}
  & $p=2$ & 3 & 5 & 7 & 11 & 13 & 17 & 19 \\
  \hline
$d=1$ & \bf 8 & \bf 5 \\
2 & \bf 10 & \bf 5 & \bf 4 \\
3 & 9 ({\bf 8}) & \bf 7 & 3 ({\bf 2})  & \bf 4 \\
4 & \bf 12 & 6 ({\bf 5}) & \bf 4 & 3 ({\bf 2}) \\
5 & 11 ({\bf 8}) & 6 ({\bf 5}) & 4 ({\bf 2})  & 3 ({\bf 2})  & \bf 4 \\
6 & 11 ({\bf 10}) & \bf 7 & \bf 4 & \bf 4 & 3 ({\bf 2})  & 4 \\
7 & 10 ({\bf 8}) & 6 ({\bf 5}) & 4 ({\bf 2}) & 4 ({\bf 2})  & 3 ({\bf 2})  & 3 ({\bf 2})  \\
8 & {\bf 14} & 6 ({\bf 5}) & \bf 4 & 3 ({\bf 2})  & 3 ({\bf 2})  & 3 ({\bf 2})  & 4 \\
9 & 13 ({\bf 8}) & \bf 9 & 4 ({\bf 2}) & \bf 4 & 3 ({\bf 2})  & 3 ({\bf 2})  & 3 ({\bf 2})  & 4* \\
10 & 13 ({\bf 10}) & 8 ({\bf 5}) & 6 & 4  ({\bf 2})  &  4* & 3 ({\bf 2})  & 3 ({\bf 2})  & 3 ({\bf 2})  \\
\end{tabular}
\caption{Brumer--Kramer bounds $B'(p,d)$ for simple maximal RM, with the bounds $B_0(p,d)$ in parentheses when different; bounds known to be sharp (resp.\ almost sharp) are bolded (resp.\ starred)}
\label{tab:bounds}
\end{table}

The proof of part (1) relies on a result about rationality fields
of modular forms.  
Denote by $\zeta_m$ a primitive $m$-th root of unity in 
$\bar \Q$.  Let $\Q(\zeta_m)^+ = \Q(\zeta_m + \zeta_m^{-1})$ be the maximal
totally real subfield of $\Q(\zeta_m)$.  Write $K_f = \Q((a_n)_n)$ for the rationality
field of a cuspidal newform $f = \sum a_n q^n$.

Let $f \in S_{2k}(N)$ be a newform.  (Here $S_*(N)$ indicates level $\Gamma_0(N)$ with trivial nebentypus).  Then \cref{prop1} states that when $v_p(N)$
is sufficiently large, $K_f \supset \Q(\zeta_{p^r})^+$, where $r$ is a certain 
function of $v_p(N)$ that grows like $v_p(N)/2$.
Consequently, if we fix the (rationality) degree $d := [K_f : \Q]$ of $f$,
then this places bounds on $v_p(N)$, precisely $v_p(N) \le B_0(p,N)$.
Applying these bounds for a weight 2 newform $f$ associated to $A$
leads to the first part of the theorem.

\cref{prop1} is not entirely novel.  Brumer \cite[Theorem 5.5]{brumer}
proved a version of it in terms of abelian varieties of GL(2)-type,
which suffices to deduce \cref{rem:GL2}. 
However, to our knowledge, \cref{rem:GL2} has not appeared in the literature.
Moreover, \cref{prop1} is sharper for $p=2$, and our proof also
allows us to say something more about the local representations at $p=3$ (see \cref{cor:sc}).

Apart from sharpening Brumer's result when $p=2$, what is new
here is an explicit formulation of the bounds $B_0(p,d)$, a direct comparison
with the Brumer--Kramer bounds, and a computational investigation of 
whether they are sharp.
In fact, \cref{prop1} tells us more than just the bounds in \cref{thm:main}.
We explicitly spell out the stronger conclusions one can make for 
$2 \le d \le 6$.

\begin{prop} \label{prop2}
Let $A$ be a $d$-dimensional simple abelian variety with maximal RM
of conductor $N^d$.  Set $K =  \End^0(A)$, which is a totally 
real number field of degree $d$. 

\begin{enumerate}[(i)]

\item Suppose $d=2$.  If $2^9 \mid N$, then $K = \Q(\sqrt 2)$.
If $5^3 \mid N$, then $K = \Q(\sqrt 5)$.  In particular,
$2^9 \cdot 5^3 \mid N$ is impossible.

\item Suppose $d=3$.  If $3^6 \mid N$, then $K = \Q(\zeta_9)^+$.  
If $7^3 \mid N$, then $K = \Q(\zeta_7)^+$.  Hence
$3^6 \cdot 7^3 \mid N$ is impossible.

\item Suppose $d=4$.  If $2^{11} \mid N$, then $K = \Q(\zeta_{16})^+$.
If $2^9 \mid N$ (resp $5^3 \mid N$), then $K$ contains $\Q(\sqrt 2)$
(resp.\ $\Q(\sqrt 5)$).  Hence if $2^9 \cdot 5^3 \mid N$, then 
$K = \Q(\sqrt 2, \sqrt 5)$.  Further, $N$ cannot be divisible by
$2^{11} \cdot 5^3$.

\item Suppose $d = 5$.  If $11^3 \mid N$, then $K = \Q(\zeta_{11})^+$.

\item Suppose $d = 6$.  If $2^9 \cdot 3^6$, $2^9 \cdot 7^3$, $3^6 \cdot 5^3$,  $5^3 \cdot 7^3$, or $13^3$ divides $N$, then $K = \Q(\sqrt 2, \zeta_9 + 
\zeta_9^{-1})$, $\Q(\sqrt 2, \zeta_7 + \zeta_7^{-1})$, 
$\Q(\sqrt 5, \zeta_9 + \zeta_9^{-1})$, $\Q(\sqrt 5, \zeta_7 + \zeta_7^{-1})$,
or $\Q(\zeta_{13})^+$, respectively.
Further, $N$ cannot be divisible by $2^9 \cdot 5^3$,
$2^9 \cdot 13^3$, $3^6 \cdot 7^3$, $3^6 \cdot 13^3$, or $7^3 \cdot 13^3$.
\end{enumerate}
\end{prop}

\begin{cor} \label{cor:gen2}
Let $C/\Q$ be a genus $2$ curve of conductor $N$.  Let $A$ be the Jacobian of $C$.
Suppose $C$ has RM (i.e., $A$ has RM).  
If $N$ is divisible by $5^6$ (resp.\ $2^{18}$), then $A$ is simple and
$\End^0(A) \simeq \Q(\sqrt 5)$ (resp.\  $\End^0(A) \simeq \Q(\sqrt 2)$).
\end{cor}

Parts (i)--(v) of \cref{prop2} also hold verbatim when $f \in S_{2k}(N)$
is a degree $d$ newform with $K=K_f$.  These statements are proved
in \cref{sec:restrRM}.

These consequences were initially quite striking to us for the following reason.  
One can  
specify a finite number of local discrete series components (and thus local 
conductors) of automorphic representations in the trace formula, 
and asymptotically count such representations (e.g., see \cite{weinstein}). 
Each of these fixed local components merely contribute
independent local densities (specified by the Plancherel measure) to
this asymptotic.
By analogy, one might guess that, when counting weight 2 newforms of
a fixed degree $d$ with a finite number of local conductors specified,
the asymptotic may be a product of local densities 
for each of these local conductors.  Indeed, the work \cite{SSW}
on counting elliptic curves suggests this is the case for $d=1$.

In contrast, there exist degree 2 newforms (or, if one prefers, genus 2 curves with RM) whose level is divisible by 125, and 
also degree 2 newforms with level divisible by 512.  A naive guess is that
each of these 
sets make up some positive proportions, say $\delta_{125}$ and $\delta_{512}$, of all degree 2 newforms,
and that the proportion of all degree 2 newforms with level divisible by 
$125 \cdot 512$ should be $\delta_{125} \cdot \delta_{512} > 0$.  
But the corollary tells us there are no such forms!  
 (In fact, \cite{CM} suggests 100\% of
all degree 2 forms may have rationality field $\Q(\sqrt 5)$, so it may be
that $\delta_{512} = 0$, but a local-global counting principle would
still suggest there should be infinitely many such forms.)
Hence \cref{prop2} and \cref{cor:gen2} say that, for $d \ge 2$,
local conductor behavior is not independent, at least in the case of 
``sufficiently wild'' ramification. 

\section{Rationality subfields for modular forms}

\begin{prop} \label{prop1}
Let $f \in S_{2k}(N)$ be a newform and fix a prime $p$.
Suppose $p^3 \mid N$ if $p$ is odd and $p^9 \mid N$
if $p=2$.   
Set  $r =  \lceil v_p(N)/2 - v_p(3)/2 \rceil - 1 - v_p(2)$.  
Then $\Q(\zeta_{p^r})^+ \subset K_f$.
In particular, $\frac 12 p^{r-1}(p-1) \mid [K_f : \Q]$.
\end{prop}

The lower bound of 3 or 9 on $v_p(N)$ in the hypothesis is the
minimum needed to get a nontrivial conclusion, except when
$p=3$ where one needs $3^6 \mid N$.

This proposition refines earlier results of Saito and Brumer.
Namely, \cite[Corollary 3.4]{saito} uses certain operators
on $S_{2k}(N)$ to obtain a weaker version of this proposition with 
$r = \lfloor v_p(N)/3 \rfloor - v_p(2)$.  In the case $2k=2$,
 \cite[Theorem 5.5]{brumer} obtains a version of this result 
 in the context of abelian varieties of $\GL(2)$-type
with $r = \lceil v_p(N)/2 - 1 - \frac 1{p-1} \rceil - v_p(2)$.\footnote{There is a typo in the $p=2$ case
 of \cite[Theorem 5.5(i)]{brumer}---the field in the $p=2$ case should read as in part (ii) of \emph{loc.\ cit.}}
Brumer remarks that his approach also applies to higher weights.
This is almost the same as our value of $r$, but is smaller by 1 when
$p=2$ and $v_2(N)$ is odd.
E.g., when $p=2$, Brumer's result only implies 
$\Q(\sqrt 2) \subset K_f$ if $v_2(N) \ge 10$, but this proposition
says $v_2(N) \ge 9$ suffices.
So this proposition should be viewed as a slight sharpening
of Brumer's result when $p=2$.

We will prove this proposition by examining local Galois
types of modular forms, following \cite{DPT}.  In fact, the proof
amounts to making some arguments in \cite{DPT} slightly more precise,
and correcting an error therein.
The basic idea is the same as Brumer's: show
$\zeta_{p^r} + \zeta_{p^r}^{-1}$ appears as a trace of an
appropriate representation.  However, our proof uses
an explicit description of possible representations, and we 
obtain slightly more information when $p=2, 3$.

In fact neither Saito's nor Brumer's result, nor the argument we
give, requires trivial nebentypus.  We assume it for simplicity
so that we can directly apply \cite{DPT}.  

\begin{proof}
Let $\pi_p$ be the smooth irreducible representation of $\GL_2(\Q_p)$
associated to $f$.  Since we assumed $v_p(N) \ge 3$, either $\pi_p$
is an irreducible principal series representation or a supercuspidal 
representation (i.e., twisted Steinberg is not possible).
For a local representation $\sigma$,
we write $c(\sigma)$ for its conductor.

First assume $\pi_p = \pi(\mu, \mu^{-1})$ is an irreducible principal
series.  Then $c(\pi_p) = 2n$, where $n = c(\mu)$.  Now
$\mu$ factors through $\Z_p^\times / (1 + p^n \Z_p^\times) \simeq
(\Z/p^n\Z)^\times$, which is cyclic if $p$ is odd and
abstractly isomorphic to $C_2 \times C_{2^{n-2}}$ if $p=2$.  Since 
$\mu |_{\Z_p^\times}$ does not factor though $1 + p^{n-1} \Z_p^\times$,
necessarily the order of $\mu |_{\Z_p^\times}$ is a multiple of
$p^{n-1-v_2(p)}$.  Now \cite[Lemma 3.1]{DPT} implies that
$\zeta_{p^r} + \zeta_{p^r}^{-1} \in K_f$, where $r = n-1-v_2(p) = v_p(N)/2 - 1 - v_2(p)$.
(This lemma is proved by looking at the images of the associated
$\ell$-adic Galois representations $\rho_{f,\lambda}$.)
This proves the proposition in the principal series case.

Now assume $\pi_p$ is supercuspidal.  All the details we will
need about supercuspidal representations are recalled or proved
in \cite[Section 2]{DPT}, with a correction noted below.  
Necessarily, $\pi_p$ is dihedral if $p$ is odd.
If $p=2$, there are non-dihedral supercuspidal representations $\pi$, but 
they have conductor exponent $c(\pi) \le 8$ (in fact $\le 7$), 
in which case the proposition is vacuous.  Thus for any
$p$ we may assume that $\pi_p$ is dihedral.  This means that $\pi = \pi_p$
corresponds to the induction $\Ind_{W_{\Q_p}}^{W_E} \theta$ of a regular character $\theta$ of the Weil group of some quadratic extension $E/\Q_p$.  
Write $n = c(\theta)$.

First suppose $E/\Q_p$ is unramified.  Then $c(\pi) = 2n$.
As explained in the proof of \cite[Lemma 2.11]{DPT}, the order of
$\theta$ is a multiple of $p^{n-1-v_2(p)}$.

Next suppose $E/\Q_p$ is ramified with $p$ odd.  Then $c(\pi) = n+1$
and $n \ge 2$.  \emph{Loc.\ cit.}\ states that $n$ must be even and
the order of $\theta$ is a multiple of $p^{n/2}$.  This is correct except
in the special case that $E \simeq \Q_3(\sqrt{-3})$, where we can only 
say that the order of $\theta$ is a multiple of $3^{n/2-1}$.  
The difference in the exponent in this case arises from difference in the 
structure of $\frako_E/\frakp_E^j$ when $E \simeq \Q_3(\sqrt{-3})$ due to the presence of extra roots of unity.  The structure of $\frako_E/\frakp_E^j$ 
is described in \cite[Theorem 2.10]{DPT} 
and implies the above
fact about the order of $\theta$.
(We remark that there is a typo in this theorem: $a$ and $b$ should be swapped in third line of data in \cite[Table 2]{DPT}, which corresponds to $E \ne \Q_3(\sqrt{-3})$ is ramified with $p \ne 2$, and gives the exponent $\frac n2$ in this case.
Neither this typo nor the above correction affect the other proofs or results in \cite{DPT}.)

Finally suppose $p=2$ and $E/\Q_2$ is ramified.  Then $c(\pi) = n + \delta$,
where $\delta \in \{ 2, 3 \}$ is the 2-adic valuation of the discriminant of $E$.  
As remarked above, when $p=2$, we may assume $c(\pi) = n+ \delta \ge 9$,
in which case $n$ is necessarily even.  When $\delta = 2$,
i.e., $E \simeq \Q_2(\sqrt{-1})$ or $\Q_2(\sqrt 3)$, then $\theta$ has
order a multiple of $2^{n/2-1}$.  If $\delta = 3$, i.e.,
$E \simeq \Q_2(\sqrt{\pm 2})$ or $\Q_2(\sqrt{\pm 6})$, then
$\theta$ has order a multiple of $2^{n/2}$.

Thus in all dihedral supercuspidal cases of relevance, the order of $\theta$ 
is a multiple of $p^r$, where $r$ is as in the statement of the proposition.  Now 
\cite[Lemma 3.2]{DPT} implies $\zeta_{p^r} + \zeta_{p^{-r}} \in K_f$,
as desired.
\end{proof}

\begin{cor} \label{cor1}
Suppose $f \in S_{2k}(N)$ is a newform of degree $d$.
Then $v_p(N) \le B_0(p,d)$, where $B_0(p,d)$ is given in \eqref{eq:B0def}.
\end{cor}

\begin{proof} 
Suppose $p \ge 5$.  Since $B_0(p,d) \ge 2$, we may assume
$v_p(N) \ge 3$.  Then \cref{prop1} implies $\frac 12 p^{r-1}(p-1) \mid d$,
where $r = \lceil \frac{v_p(N)} 2 \rceil - 1 \ge 1$.  Hence $\frac {p-1}2 \mid d$
and $\lceil \frac{v_p(N)} 2 \rceil  -  2 \le v_p(d)$.  Thus $v_p(N) \le B_0(p,d)$.

The same reasoning gives the result for $p = 2$ and $p=3$.
\end{proof}

We point out one more consequence of the proof of \cref{prop1}.
Given a newform $f \in S_{2k}(N)$, denote by $\pi_p(f)$ the
local admissible representation of $\GL_2(\Q_p)$ associated to $f$.
An algorithm for determining $\pi_p(f)$ is given in \cite{LW}.  
The main difficulty is distinguishing 
supercuspidal representations, which \cite{LW} carries out using the 
cohomology of the modular curve.
Simply knowing $v_p(N)$ places many restrictions on the possibilities
for $\pi_p(f)$---in particular if $v_p(N) \ge 3$ is odd, $\pi_p(f)$ must be
supercuspidal.    The following gives an elementary
way to partially distinguish supercuspidal components using $K_f$ when $p=3$.

\begin{cor} \label{cor:sc}
Let $f \in S_{2k}(N)$ be a newform and suppose
$3^{2m+1} \parallel N$ with $m \ge 1$.  
If $\Q(\zeta_{3^m})^+ \not \subset K_f$, then $\pi_3(f)$
is a supercuspidal representation dihedrally induced from 
$E = \Q_3(\sqrt{-3})$.
\end{cor}

\begin{proof}
Since $c(\pi_3(f))=2m+1 \ge 3$ is odd, $\pi_3(f)$ must be a dihedral supercuspidal induced from a character $\theta$ along a ramified quadratic 
extension $E/\Q_3$.  Now the corollary follows from the difference in exponents for our lower bounds for the order of $\theta$ in the cases $E \not \simeq \Q_3(\sqrt{-3})$
and $E \simeq \Q_3(\sqrt{-3})$.
\end{proof}

For instance, up to Galois conjugacy, $S_2(243)$ has 2 rational newforms,
1 newform with rationality field $\Q(\sqrt 3)$, 1 newform with rationality field
$\Q(\sqrt 6)$, and 2 newforms with rationality field $\Q(\zeta_9)^+$.  
The 4 Galois orbits of newforms with rationality fields of degree $\le 2$ 
must all have local components $\pi_3$ being dihedral supercuspidal 
representations induced from $\Q_3(\sqrt{-3})$.

\section{Local conductor bounds for modular abelian varieties} \label{sec3}

In this section we will prove the results stated in the introduction.

Given a weight 2 newform $f$ for $\Gamma_0(N)$ or $\Gamma_1(N)$,
there is an associated simple abelian variety $A_f$ which is a subquotient
of $J_0(N)$ or $J_1(N)$.  
%Denote by $K_f$ the rationality field of $f$.
Moreover, $\dim A_f = [K_f:\Q]$, $\End^0(A_f) \simeq K_f$ and $L(s,A) = 
\prod L(s, f^\sigma)$, where $f^\sigma$ runs over the Galois conjugates of
$f$.  We say an abelian variety $A/\Q$ corresponds to a weight 2 newform
$f$ if it is isogenous to $A_f$ (over $\Q$), in which case we say $A$ is
modular.  Recall that isogenies preserve endomorphism algebras.

\begin{lem} \label{lem:modularity} 
Suppose $A/\Q$ is a simple $d$-dimensional abelian
variety with maximal RM.  
Then $A$ corresponds to a newform $f \in S_2(N)$, where $N^d$ is the
conductor of $A$.
\end{lem}

This lemma together with \cref{cor1} implies part (1) of \cref{thm:main}.

As we do not know a reference that explicitly concludes the newform $f$
must have trivial nebentypus, we provide a proof.

\begin{proof}
Since Serre's conjecture is known, \cite[Theorem 4.4]{ribet:korea} implies
that $A$ is a simple factor of $J_1(N)$.  Thus $A$ is isogenous to
an abelian variety $A_f$ associated to a newform $f \in S_2(M, \eps)$
for some $M \mid N$ and nebentypus $\eps$.  Comparing conductors shows
$M = N$.  Now the rationality field of $f$ equals $K = \End_0(A)$, and thus
contains the values of $\eps$.  Hence $\eps$ is quadratic.  If $\eps = 1$,
then $f$ is a newform in $S_2(N)$, as desired.

We will use several facts about CM forms, all of which can be found in  \cite{ribet:survey}.
Suppose $\eps$ is nontrivial. 
Then $f$ has CM by $\eps$, i.e., $f  \otimes \eps = f$.  We claim this is 
impossible because $f$ has weight 2.

More generally, take a newform $f \in S_k(N, \eps)$ of any weight $k \ge 2$,
and assume $f$ has CM by $\eps$.
Then $f$ is induced from a Grossencharacter $\psi$ of an imaginary
quadratic field $F$, and $f$ has CM by the quadratic Dirichlet character
$\chi_F$ attached to $f$.  Since the character of CM is unique, $\eps = \chi_F$.  This forces $\eps$ to be odd, as $F$ is imaginary quadratic.
But because $\eps$ is also the nebentypus character, we must have
$\eps(-1) = (-1)^k$.  Hence $k$ is odd, proving the claim.
\end{proof}

\subsection{Comparison with Brumer--Kramer bounds} \label{sec:BK}
Suppose $A/\Q$ is a $d$-dimensional simple factor of the new part 
of $J_0(N)$.  Thus $A$ has conductor $N^d$.  Put $e = v_p(N)$.
Brumer and Kramer's  conductor exponent bounds for abelian varieties 
\cite[Theorem 6.2]{brumer-kramer} states that
\[ d v_p(N) \le B(p,d) := 2d +   p t + (p-1) \lambda_p(t), \]
where $t =  \lfloor 2d/(p-1) \rfloor$ and $\lambda_p$ is defined by
\[ \lambda_p(m) = \sum_{i=0}^s i c_i p^i, \quad \text{for } m = \sum_{i=0}^s c_i p^i,
\quad 0 \le c_i < p. \]
Let us rewrite this as 
\[ v_p(N) \le B'(p,d) := 2 + \lfloor \frac{pt + (p-1) \lambda_p(t)}d \rfloor.  \]
Thus $B'(p,d)$ can be thought of as the Brumer--Kramer bound for abelian varieties of GL(2)-type.
We want to compare $B'(p,d)$ with $B_0(p,d)$.

\begin{lem}
First suppose $p \ge 5$.   When $p \ge d$, $B'(p,d)$ is given by
\[ B'(p,d) = \begin{cases}
2 & p > 2d +  1 \\
4 & p = 2d + 1 \text{ or } p= d \\
3 & d + 1 < p  <  2d + 1
\end{cases} \]
If $5 \le p \le d$, then $B'(p,d) \ge 3$.

When $p=3$, we have $B'(3,1) = B'(3,2) = 5$ and $B'(3,d) \ge 6$ for $d \ge 3$.

When $p = 2$, we have $B'(2,1) = 8$, $B'(2,2) = 10$, $B'(2,3) = 9$ and
$B'(2,d) \ge 9$ for $d \ge 4$.

For any $p$, if $(p-1) \mid 2d$, then 
$B'(p,d) \ge 4 + 2v_p(d) + 4v_2(p) + v_2(3)$, with equality when
$\frac{2d}{p^{v_p(2d)}(p-1)} < p$.
\end{lem}

\begin{proof}
Note that $\lambda_p(t) = 0$ if and only if $t < p$, i.e., $2d < p(p-1)$.

First suppose $p \ge 5$.
Since $t = 0 \iff p > 2d+1$, we see $B'(p,d) = 2 = B_0(p,d)$ when $p > 2d+1$.
We also note that if $p=2d+1$, so $t=1$, then $B'(p,d) = 4 = B_0(p,d)$.
If $d + 1 < p < 2d + 1$, so $t=1$, then $B'(p,d) = 3 > B_0(p,d) = 2$.
When $p=d$, so $t=2$, $B'(p,d) = 4 > B_0(p,d) = 2$.

Next suppose $5 \le p < d$.  Then $t \ge 2$. Since $t + 1 > \frac{2d}{p-1}$,
we have $(p-1) t + (p-1) > 2d$, and thus $pt > 2d + 1 + t - p > d + 1$.  Hence
$\frac{pt}d > 1$ and $B'(p,d) \ge 3$.

For $p = 3$, note that $t=d$ so $B'(p,d) = 5 + \lfloor \frac{2 \lambda_3(d)}d \rfloor$.
Since we always have $\lambda_p(m) \ge m - p + 1$, one gets $2 \lambda_3(d) \ge 3$ for all $d \ge 3$.  This gives the $p=3$ statement.

When $p = 2$, we have $t=2d$ and $B'(p,d) = 6 + \lfloor \frac{\lambda_2(2d)}d \rfloor$.  One can compute $B'(2,d)$ case-by-case for $d \le 3$.  
If $s = \lfloor \log_2(d) \rfloor$, then $2^{s+1} > d$ so 
$\lambda_2(2d) \ge (s+1) 2^{s+1} > (s+1)d \ge 3d$.
This implies the $p=2$ lower bound.

Lastly, consider any $p \ge 2$.
Suppose $2d = a p^m (p-1)$ for some $m \ge 0$ and integer $a$ coprime to $p$.
Then $t = a p^m$ so $\frac{pt}d = 2 + \frac 2{p-1}$.  
If $a < p$, then $\lambda_p(t) = m a p^m$, and thus 
$\frac{(p-1)\lambda_p(t)}d = 2m = 2(v_p(d) + v_2(p))$  
If $a > p$, then $\lambda_p(t) > m a p^m$.
This proves the assertion when $(p-1) \mid 2d$.
\end{proof}

The lower bounds are not intended to be optimal, but merely sufficient to 
conclude the following, which implies part (2) of \cref{thm:main}.

\begin{cor}
We have $B_0(p,d) \ge B'(p,d)$.  

This is an equality if (i) $d \le 2$; (ii) $p \ge 2d+1$; or
(iii) if $2d = ap^m(p-1)$ for some $m \ge 0$ and $a < p$.

We have a strict inequality $B_0(p,d) < B'(p,d)$ if 
(i) $\max\{ 5, d \} \le p < 2d + 1$; (ii) $5 \le p < d$ and $(p - 1) \nmid d$;
or (iii) $p \in \{ 2, 3 \}$, $d > p$, and $p \nmid d$.
\end{cor}

\begin{proof} This follows from the previous lemma.  The cases $d \le 2$
are more easily seen from examining \cref{tab:bounds}.
\end{proof}

\subsection{Sharpness of bounds}
Here we consider when the bounds $v_p(N) \le B_0(p,d)$ of \cref{thm:main}
are sharp.  Note that
once there exists, say, a weight 2 newform $f$ of degree $d$, one can replace
$f$ with a quadratic twist to ensure $v_p(N) \ge 2$.
Thus we will look for examples where $v_p(N) = B_0(p,d) > 2$ is attained
for $d \le 10$.  This will justify the bolded entries in \cref{tab:bounds}.

First we searched the LMFDB \cite{LMFDB} for weight 2 newforms with trivial
nebentypus which attain such bounds.
The LMFDB contains data on all weight 2 newforms of level up to 10000.
This data shows the bounds are sharp when (i) $d \le 4$; 
(ii) $d = 5$ and $p \ne 11$; (iii) $d = 6$ and $p \ne 13$; 
(iv) $d = 7$ and $p \ne 2$; (v) $d=8$ and $p \ne 2, 17$;
(vi) $d= 9$ and $p \ne 2, 3, 19$; and (vii) $d=10$ and $p \ne 2, 5, 11$.

Note that $11^4 = 14641$ is already outside of the LMFDB range, so the
LMFDB data is necessarily insufficient to search for $v_p(N) = B_0(p,d)$
when $p = 2d+1 \ge 11$.  Similarly, $2^{14}$, $3^9$ and $5^6$ are also outside
of the LMFDB range.
%The other limitation in searching the LMFDB is that if  $N$ has many divisors, including large prime power factors, then there are many local inertial types and Atkin--Lehner signs, which prevents large degree forms appearing if $N \le 10000$.  E.g., for $N = 2^8 \cdot 3 \cdot 13 = 9984$, the newspace has dimension 192 but no newform with degree $> 8$ appears.

For the other cases in \cref{tab:bounds}, we used 
Magma \cite{magma} to compute newform decompositions (i.e., degrees
of all newforms) for $S_2(N)$
for various $N > 10000$.  Andrew Sutherland has also done some such
calculations, and at least some of our data is contained in his.  In particular,
we found the following, which justifies the remaining bolded entries in 
\cref{tab:bounds}.  We indicate approximate runtimes and RAM usage
for each newform decomposition calculation in parentheses.

\begin{itemize}
\item
$N = 12032 = 2^8 \cdot 47$ has a degree 7 newform (32s, 245MB)

\item
$N = 14592 = 2^8 \cdot 57$ has a degree 9 newform (4min, 1.3GB)

\item
$N = 11264 = 2^{10} \cdot 11$ has a degree 10 newform (52s, 325MB)

\item
$N = 16384 = 2^{14}$ has a degree 8 newform (65s, 444MB)

\item
$N = 19683 = 3^9$ has a degree 9 newform (4min, 763MB)

\item 
$N = 14641 = 11^4$ has a degree 5 newform (9min, 1.5GB)
\end{itemize}

The remaining cases to check are when $d=10$ and $p=5, 11$,
or when $d = 6,8,9$ and $p=2d+1$.  In these cases, we searched for newforms
where $v_p(N) = B_0(p,d)$ or $v_p(N) = B_0(p,d)-1$ is attained, to the
extent we could with moderate computational resources.  
We did not find any levels outside of the LMFDB range where
either of these bounds is attained, but we summarize our attempts.
We remark that these computations tend to be more accessible for levels of the form
$2^m p^r$ than for $q \cdot p^r$ where $q$ is an odd prime of comparable size
to $2^m$.  

For $p=5$, we computed newform decompositions for levels $N = 5^6 \cdot M$
where $M = 2, 4, 8$, as well as $N = 5^5 \cdot M$ for $M \le 12$ and 
$M = 14, 16, 18$.  The decomposition for $N = 125000 = 8 \cdot 5^6$ took over 7 hours and used 46 GB of RAM.
None of these levels have degree $d=10$ forms, though
there are degree 20 forms in levels $18750 = 6 \cdot 5^5$, $28125 = 9 \cdot 5^5$, $50000 = 16 \cdot 5^5$ and $56250 = 18 \cdot 5^5$.
There are degree 10 newforms in the LMFDB with level 
$N=8750 = 14 \cdot 5^4$, so at least $v_{5}(N) = B_0(5,10)-2$ is attained.

For $p=11$, we computed newform decompositions for levels $N = 11^4 \cdot M$
for $M = 1, 2, 4, 8$.  None of these levels have degree $d=10$ forms.
There are however many degree 10 forms in the LMFDB with
$v_{11}(N) = 3$, e.g., there is one in level $1331 = 11^3$, so at least
$v_{11}(N) = B_0(11,10)-1$ is attained.

For $p = 13$, we computed newform decompositions for levels $N = 13^4 \cdot M$
with $M = 1, 4$, as we well as levels $N = 13^3 \cdot M$ for $M \le 16$ and
$M = 18, 20, 24, 32$. The decomposition for $N = 114244 = 4 \cdot 13^4$ took 11 hours and used 61 GB of RAM.
None of these levels have degree 6 newforms, though there are degree
12 newforms in levels $35152 = 16 \cdot 13^3$, $39546 = 18 \cdot 13^3$ and $70304 = 32 \cdot 13^3$.

For $p = 17$, we computed newform decompositions for levels $N = 17^3 \cdot M$
for $M \le 9$ and $M = 12, 16$.
(For level $7 \cdot 17^3$, we had to use
an alternate method of factoring Hecke polynomials.)
None of these levels have degree 8 newforms, though there are degree
16 newforms in level $44217 = 9 \cdot 17^3$.

For $p=19$, the LMFDB lists degree 9 newforms in level $6859 = 19^3$.

\subsection{Restrictions on RM fields} \label{sec:restrRM}

Let $A$ be simple abelian variety with maximal RM of conductor $N^d$,
where $d = \dim A$.
A further application of \cref{prop1} (combined with \cref{lem:modularity})
is that large powers of $p$ dividing $N$ force restrictions on the 
endomorphism algebra $K = \End^0(A)$, sometimes determining 
it completely.  Moreover, different primes can interact in this 
phenomenon, creating ``global'' restrictions on local conductor exponents.

We can think of this as follows.  Say $v_p(N)$ is large, and
let $r$ be as in \cref{prop1}.  This forces $K = \End^0(A)$
to contain $\Q(\zeta_{p^r})^+$, and leaves only 
$d' = d / (\frac 12 p^{r-1}(p-1))$ ``degrees of freedom'' inside $K$
for other cyclotomic subfields.  So if $q$ is another prime dividing $N$,
we have the stronger conductor bound $v_q(N) \le B_0(q, d')$, since
the cyclotomic fields $\Q(\zeta_{p^r})$ and $\Q(\zeta_{q^s})$ are disjoint
for $q \ne p$.

Applying this reasoning case-by-case for $2 \le d \le 6$ proves \cref{prop2}.

The $d=2$ case of this proposition immediately implies \cref{cor:gen2}
once one knowns $A = \mathrm{Jac}(C)$ is simple.
Indeed, if it were not, it would be a product of elliptic curves $E_1$ and
$E_2$, and then $N$ is the product of the conductors of $E_1$ and $E_2$.
But by local conductor bounds for elliptic curves, this is impossible
if $v_2(N) > 16$ or $v_5(N) > 4$.

\subsection{Abelian varieties of GL(2)-type} \label{sec:GL2}
Now suppose $A$ is a $d$-dimensional simple abelian variety of GL(2)-type,
not necessarily with maximal RM.  Then $A$ is isogenous to a newform $f \in S_2(N,\eps)$.  Suppose $p, r$ are as in \cref{prop1}.
When $p$ is odd or $p=2$ and $v_2(N)$ is even, 
\cite[Theorem 5.5]{brumer} tells us that  $\Q(\zeta_p^{r})^+ \Q(\eps) \subset K_f$.
If $p=2$ and $v_2(N)$ is odd, Brumer's theorem says
$\Q(\zeta_p^{r-1})^+ \Q(\eps) \subset K_f$.
Thus using \cite[Theorem 5.5]{brumer} in place of \cref{prop1} in the proof of
\cref{thm:main}(1) yields the statement in \cref{rem:GL2}.

\subsection*{Acknowledgements}
We thank Armand Brumer and Alex Cowan for discussions which led
to this note.  We also thank Ariel Pacetti for clarifications on \cite{DPT}.
Andrew Sutherland kindly explained how to compute newform decompositions
in Magma, and shared with us some of his data.
We are grateful to the referee for numerous suggestions and comments which substantially improved the manuscript.
Some of the computing for this project was performed at the OU Supercomputing Center for Education \& Research (OSCER) at the University of Oklahoma (OU).

%%%%%%%%%%%%%%%%%%%%%%%%%%%%%%%%
%%%%%%%%%%%%%%%%%%%%%%%%%%%%%%%%
%
%  REFERENCES
%
%%%%%%%%%%%%%%%%%%%%%%%%%%%%%%%%
%%%%%%%%%%%%%%%%%%%%%%%%%%%%%%%%

\end{document}